\documentclass[12pt]{amsart}
\usepackage{amsmath,amssymb,amsfonts,amsthm,latexsym}

\usepackage{color}
\usepackage[usenames,dvipsnames,svgnames,table]{xcolor}

\DeclareMathAlphabet{\curly}{U}{rsfs}{m}{n}

\newtheorem{thm}{Theorem}
\newtheorem{cor}[thm]{Corollary}
\newtheorem{lem}{Lemma}[section]

\newtheorem{remark}{Remark}

\makeatletter
\renewcommand{\pmod}[1]{\allowbreak\mkern7mu({\operator@font mod}\,\,#1)}
\makeatother

\newcommand{\dalign}[1]{\[\begin{aligned} #1 \end{aligned}\]}

\newcommand{\ZZ}{{\mathbb Z}}

\newcommand{\RR}{{\mathbb R}}

\newcommand{\bb}{{\mathbf b}}

\newcommand{\LL}{\curly L}

\newcommand{\BB}{\curly B}

\newcommand{\TT}{\curly T}
\newcommand{\PP}{\curly P}

\newcommand{\Vol}{\operatorname{Vol}}   

\newcommand{\g}{\ensuremath{\gamma}}

\newcommand{\lam}{\ensuremath{\lambda}}
\newcommand{\eps}{\ensuremath{\varepsilon}}

\newcommand{\bx}{{\ensuremath{\boldsymbol{\xi}}}}
\newcommand{\cE}{\mathcal{E}}
\newcommand{\cF}{\mathcal{F}}
\newcommand{\cA}{\mathcal{A}}
\newcommand{\cB}{\mathcal{B}}
\newcommand{\cR}{\mathcal{R}}

\newcommand{\flr}[1]{{\ensuremath{\left\lfloor {#1} \right\rfloor}}}

\newcommand{\pfrac}[2]{{\left(\frac{#1}{#2}\right)}}

\newcommand{\be}{\begin{equation}}
\newcommand{\ee}{\end{equation}}
\newcommand{\benn}{\begin{equation*}}   
\newcommand{\eenn}{\end{equation*}}

\renewcommand{\AA}{\curly A}
\renewcommand{\(}{\left(}
\renewcommand{\)}{\right)}

\renewcommand{\ge}{\geqslant}
\renewcommand{\le}{\leqslant}

\numberwithin{equation}{section}

\newcommand{\order}{\asymp}

\newcommand{\ssum}[1]{\sum_{\substack{#1}}}  

\newif\ifdraft

\begin{document}

\title{Rough integers with a divisor in a given interval}

\author{Kevin Ford}

\address{Department of Mathematics, 1409 West Green Street, University
of Illinois at Urbana-Champaign, Urbana, IL 61801, USA}
\email{ford126@illinois.edu}

\date{\today}
\thanks{2000 Mathematics Subject Classification: Primary 11N25;
  Secondary 62G30}
\thanks{Research supported by National Science Foundation grant
DMS-1802139}

\begin{abstract} We determine, up to multiplicative constants, the number of integers $n\le x$ that have no prime factor $\le w$
and a divisor in $(y,2y]$.  Our estimate is uniform in $x,y,w$.
We apply this to determine the order of the number of
distinct integers in the $N\times N$ multiplication table which are free of
prime factors $\le w$, and the number of
distinct fractions of the form $\frac{a_1a_2}{b_1b_2}$
with $1\le a_1 \le b_1\le N$ and $1\le a_2\le b_2 \le N$.  
\end{abstract}

\maketitle

%
%
%
\section{Introduction}\label{sec:intro}
%
%
%

In the paper \cite{F}, the author established the order of growth
of
$H(x,y,z)$, the number of integers $n\le x$ which have a divisor in
the interval $(y,z]$, for all $x,y,z$.  An important special case is
\be\label{Hxy2y}
H(x,y,2y) \asymp \frac{x}{(\log y)^{\cE} (\log_2 y)^{3/2}} \qquad
(3\le y\le \sqrt{x}),
\ee
where
\[
\cE = 1 - \frac{1+\log_2 2}{\log 2} = 0.086071332\ldots.
\]
A shorter, more direct proof of the 
order of magnitude bounds in the special case \eqref{Hxy2y}
is given in \cite{F2}.
More on the history of estimations of $H(x,y,z)$, further applications
and references may be found in \cite{F}.

A number of recent aplications have required similar bounds, but where the underlying set of integers $n$ is restricted to
a special set, e.g. the set of shifted primes (\cite[Theorem 6,7]{F}, \cite{KouP}) or
the values of a polynomial \cite{Erdos52,ES90,T90a,T90b,FKSY}.
More generally, we define 
\[
H(x,y,z;\cA) = | \{ n\le x, n\in \cA : d|n \text{ for some } d\in (y,z] \}|.
\]
Another natural set to consider is $\cR_w$, the set of integers
with no prime factor $p\le w$; called \emph{$w-$rough numbers} 
by some authors.
Here we determinte the exact order of growth of $H(x,y,2y;\cR_w)$
 for all $x,y,w$; the more general quantity  $H(x,y,z;\cR_w)$
 can be estimated by similar methods, although there are many cases depending on the relative size of the parameters $w,x,y,z$.

\begin{thm}\label{mainthm}
  Suppose that $4\le  y\le \sqrt{x}$, $4\le w\le y/8$ and write\footnote{The notation $\log_2 x$ stands for $\log\log x$.}
  $\delta = \frac{\log_2 w}{\log_2 y}$. 
  \begin{enumerate}
  \item[(i)] When $1-1/\log 4 \le \delta \le 1$ we have
  \[
  H(x,y,2y;\cR_w) - H(x/2,y,2y;\cR_w) \gg \frac{x}{\log^2 w} \gg H(x,y,2y;\cR_w).
  \]
\item[(ii)] When $0\le \delta < 1-1/\log 4$, we have  
  \[
H(x,y,2y;\cR_w) - H(x/2,y,2y;\cR_w) \gg
x \delta B(w,y) (\log y)^{-\cE + \frac{\log(1-\delta)}{\log 2}}
\gg H(x,y,2y;\cR_w),
\]
where 
\[
B(w,y) =  \min(1,(\log_2 y)^{-1/2}((1-\delta)\log 4 - 1)^{-1}) .
\]
  \end{enumerate}
\end{thm}

\begin{remark}
Some special cases are worth noting.  From Theorem \ref{mainthm} we 
have
\[
x \delta B(w,y) (\log y)^{-\cE + \frac{1-\delta}{\log 2}} \order
\begin{cases} 
\frac{x \log_2 w}{(\log_2 y)^{3/2}}  (\log y)^{-\cE + \frac{\log(1-\delta)}{\log 2}} & (\delta\le 1-\tfrac{1}{\log 4}-\eps, \eps>0 \text{ fixed}) \\
\frac{x \log_2 w}{(\log y)^{\cE} (\log w)^{1/\log 2} (\log_2 y)^{3/2}} &
(\log_2 w \le \sqrt{\log_2 y}).
\end{cases}
\]
\end{remark}

\begin{remark}
When $y>\sqrt{x}$, one can obtain similar results by using the
duality $d|n \iff (n/d)|n$.  That is, if $x/2<n\le x$, then
$d|n$ with $y<d\le 2y$ is equivalent to $d'|n$ with $d' \order x/y$.
\end{remark}

\medskip

We illustrate the utility of Theorem \ref{mainthm} with two applications.  The first is related to
the well-know multiplication table problem of Erd\H os
\cite{Erdos55, Erdos60}, which
asks for estimates on the number, $M(N)$, of distinct
integers in an $N\times N$ multiplication table.
In \cite{F} the author proved, using \eqref{Hxy2y}, that
\be\label{MN}
M(N) \order \frac{N^2}{(\log N)^{\cE} (\log_2 N)^{3/2}}.
\ee
More generally, consider the restricted multiplication table
problem of bounding $M(N;\cA)$, the number of distinct entries in an $N\times N$ multiplication table that belong to 
the set $\cA$. For example, when $\lambda\ne 0$ is fixed and 
$\cA = \{ p+\lambda : p \text{ prime}\}$, the order of
$M(N;\cA)$ was determined in \cite[Theorem 6]{F} (upper bound)
and \cite{KouP} (lower bound).

 Observe that $M(N;\cR_w)=1$ when $w\ge N$.

\begin{cor}
Uniformly for $4\le w\le N/2$, we have
\[
M(N;\cR_w) \order \begin{cases}
\frac{N^2}{\log^2 w} & \text{ if } \log w \ge (\log N)^{1-1/\log 4} \\
N^2 \delta B(w,N) (\log N)^{-\cE+\frac{\log(1-\delta)}{\log 2}} &
\text{ if } \log w = (\log y)^{\delta}, \delta\le 1-\frac{1}{\log 4}.
\end{cases}
\] 
\end{cor}

\begin{proof}
If $\sqrt{N} < w \le N/2$, then $M(N;\cR_w)$ counts 
entries in the multiplication table which are 
 primes in $(w,N]$ or the product of two such primes. 
 The desired bounds follow.
If $4\le w\le \sqrt{N}$, we use the inequalities
\[
H\( \frac{N^2}{4}, \frac{N}{4}, \frac{N}{2}; \cR_w\) \le M(N;\cR_w) \le \sum_{k\ge 0} H\( \frac{N^2}{2^k}, \frac{N}{2^{k+1}},\frac{N}{2^k};\cR_w\).
\]
The proof is easy: consider $ab \in \cR_w, a\le N$ and $b\le N$.
If $\frac{N}{4} < a\le \frac{N}{2}$ and $ab\le \frac{N^2}{4}$,
 then $b\le N$ and this proves the lower bound.  The upper bound 
 comes from taking $\frac{N}{2^{k+1}} < a \le \frac{N}{2^k}$
 for some non-negative integer $k$.
 The desired bound for $M(N;\cR_w)$ now follow from Theorem \ref{mainthm}, since we have $H(x,y,2y;\cR_w) \order x f(y,w)$
 where $f(u,w) \order f(y,w)$ for $\log u \order \log y$.
\end{proof}

Next, we consider the ''Farey fraction multiplication table''.
Let $\cF_N$ of Farey fractions of order $N$, i.e., 
\[
\cF_N = \left\{ \frac{a}{b} : 1\le a\le b\le N, (a,b)=1  \right\}.
\]
In private conversation, Igor Shparlinski asked the author
 about the
size of the product set $\cF_N \cF_N$
(in general, for sets $\cA,\cB\in \ZZ$, $\cA \cB$ denotes the 
product set $\{ab: a\in \cA,b\in \cB\}$).

\begin{cor}\label{Farey}
We have
\[
M(N)^2 \ll |\cF_N \cF_N| \le M(N)^2.
\]
Consequently, by \eqref{MN}, we have
\[
 |\cF_N \cF_N|  \order \frac{N^2}{(\log N)^{\cE}(\log_2 N)^{3/2}}.
\]
\end{cor}

\begin{proof}
The upper bound is trivial, and thus the real work is on the lower bound.  We achieve this by placing restrictions on the 
fractions, firstly by putting them in dyadic intervals and 
secondly by removing those elements divisible by small primes.
To this end, define
\[
\cA_N = \{n : N/2\le n\le N\}, \quad
\cA_N^{(w)} = \cA_N \cap \cR_w.
\]
Let $w$ be a large, fixed constant.
A simple inclusion-exclusion argument yields (here $p$ denotes a prime in the sums)
\dalign{
 |\cF_N \cF_N| &\ge \left| \left\{ \frac{a_1a_2}{b_1b_2} : 
 a_1,a_2 \in \cA_{N/2}^{(w)}; b_1,b_2\in \cA_N^{(w)}; (a_1a_2,b_1b_2)=1 
 \right\}\right| \\
 &\ge |\cA_{N/2}^{(w)} \cA_{N/2}^{(w)} | \cdot |\cA_{N}^{(w)} \cA_{N}^{(w)} |-
 \sum_{w<p\le N/2}  |\cA_{N/2}^{(w)} \cA_{N/2p}^{(w)} |\cdot |\cA_{N}^{(w)} \cA_{N/p}^{(w)} |\\
 &\ge  |\cA_{N/2}^{(w)} \cA_{N/2}^{(w)} | \cdot |\cA_{N}^{(w)} \cA_{N}^{(w)} |-
 \sum_{w<p\le N/2}  |\cA_{N/2} \cA_{N/2p} |\cdot |\cA_{N} \cA_{N/p} |.
}
It is clear that for $M\le N$ we have
\[
 |\cA_{N} \cA_{M} | \le H(MN,M/2,M)
\]
and we deduce from \eqref{Hxy2y} that
\dalign{
\sum_{w<p\le N/2}  |\cA_{N/2} \cA_{N/2p} |\cdot |\cA_{N} \cA_{N/p} | &\ll\sum_{p>w} \frac{N^4}{p^2(\log (N/p))^{2\cE}(\log_2 (N/p))^3}  \ll \frac{M(N)^2}{w\log w}.
}
We also have the lower bound
\[
 |\cA^{(w)}_{N} \cA^{(w)}_{M} | \ge H(MN,M/2,M;\cR_w)-H(MN/2,M/2,M;\cR_w).
\]
 It follows that
\begin{multline}\label{FNH}
|\cF_N \cF_N| \ge \Big( H(\tfrac{N^2}{4},\tfrac{N}{4},\tfrac{N}{2};\cR_w) -  H(\tfrac{N^2}{8},\tfrac{N}{4},\tfrac{N}{2};\cR_w) \Big)
\Big( H(N^2,\tfrac{N}{2},N;\cR_w) -   H(\tfrac{N^2}{2},\tfrac{N}{2},N;\cR_w)\Big) - \\
- O\pfrac{N^4}{(\log N)^{2\cE} (\log_2 N)^{3} (w\log w)}.
\end{multline}
Inserting Theorem \ref{mainthm} into the estimate \eqref{FNH},
and taking $w$ to be a sufficiently large constant, we obtain the lower bound
in Corollary \ref{Farey}.
\end{proof}

\medskip
\noindent
\subsection{Notation} 
Let $\tau(n)$ be the number of positive divisors of $n$, 
and $\tau(n;y,z)$ denotes the number of divisors of $n$ within 
the interval $(y,z]$.
Let $\omega(n)$ be the number of distinct prime
divisors of $n$.  Let
$P^+(n)$ be the largest prime factor of $n$ and let $P^-(n)$ be the 
smallest prime factor of $n$.  Adopt the notational conventions $P^+(1)=0$ and
$P^-(1)=\infty$.  Constants implied by $O$, $\ll$ and $\asymp$
are absolute. The notation $f \asymp g$ means $f\ll g$ and $g\ll f$.  The symbol $p$ will always denote a prime.
Lastly, $\log_2 x$ denotes $\log\log x$.

\subsection{Heuristics}\label{sec:heuristic}
Here we give a short heuristic argument to justify the
formulas in Theorem \ref{mainthm}.  This is similar to the heuristics givin in \cite{F,F2}.

Write $n=n'n''$, where $n'$ is composed only of primes in $(w,2y]$ and
$n''$ is composed only of primes $>2y$.
For simplicity, assume $n'$ is squarefree and $n'\le y^{100}$.
Assume for the moment that the set $D(n')=\{ \log d : d|n' \}$ is
approximately uniformly distributed in $[0,\log n']$.  If $n'$ has $k$ prime factors, then $\tau(n')=2^k$ and we thus expect that
$\tau(n',y,2y)\ge 1$ with probability about
\[
\min\( 1, \frac{2^k}{\log y} \).
\]
 This expression changes behavior at $k=k_0 := \flr{\frac{\log_2 y}{\log 2}}$.
The number of $n\le x$ with $n'\in \cR_w$ and $\omega(n')=k$ is of size
\[
\frac{x}{\log y} \frac{(\log_2 y-\log_2 w)^k}{k!},
\]
and we obtain a heuristic estimate for $H(x,y,2y;\cR_w)$ of order
\[
\frac{x}{\log^2 y}
\Bigg[ \sum_{k\le k_0} \frac{(2\log_2 y-2\log_2 w)^k}{k!}+
(\log y) \sum_{k\ge k_0} \frac{(\log_2 y-\log_2 w)^k}{k!} \Bigg].
\]
The first sum always dominates, since the
second sum is dominated by the first summand
($k_0$ is always much larger than $\log_2 y-\log_2 w$).
The behavior of the first sum over $k$ depends on the relative sizes of
$k_0$ and $2\log_2 y-2\log_2 w$.  If $k_0 > 2\log_2 y-2\log_2 w$, that is, $\log w \ge (\log y)^{1-1/\log 4}$, the 
first contains the ``peak'' and we obtain
\[
H(x,y,2y;\cR_w) \approx \frac{x}{\log^2 y} e^{2\log_2 y-2\log_2 w} =
\frac{x}{\log^2 w}.
\]
For smaller $w$, we are summing the left tail of the Poisson distribution and standard bounds (see e.g. Lemma \ref{Norton} below)  yield 
\[
H(x,y,2y;\cR_w) \approx x B(y,w) (\log y)^{-\cE + \frac{\log(1-\delta)}{\log 2}}.
\]
This latter expression is too large by a factor $1/\delta$, and
this  stems from the uniformity
 assumption about $D(n')$, which 
 turns out to be false for all but a proportion $\delta$ of
 these integers.  Fluctuations in the distribution of 
 the prime factors of $n'$ lead to clustering of the divisors;
 more details can be found in \cite{F,F2}.
 As in \cite{F,F2}, we really should be considering 
 those $n'$ which have nicely distributed divisors,
 and a useful measure of how nicely distributed the divisors are
 is the function
\[
L(a)=\text{meas} \LL(a), \qquad \LL(a)=\bigcup_{d|a} [-\log 2+\log d,\log d).
\]
Adjusting our heuristic, we see that the
probability that $\tau(n',y,2y)\ge 1$ should be about
 $L(n')/\log y$, which is $\gg 1/\log y$ on a set of $n'$
 of density $\delta$.

%
%
%

\section{Preliminaries}\label{sec:prelim}
\begin{lem}[{\cite[Lemma 3.1]{F2}}]\label{Lbounds}
We have
\begin{enumerate}
\item  $L(a) \le \min(\tau(a)\log 2, \log 2+\log a)$;
\item  If $(a,b)=1$, then $L(ab) \le \tau(b) L(a)$;
\item  If $p_1 < \cdots < p_k$, then
$$
L(p_1\cdots p_k) \le\min_{0 \le j\le k} 2^{k-j} (\log(p_1\cdots p_j)+\log 2).
$$
\end{enumerate}
\end{lem}

Let $\PP(a,b)$  be the set of all squarefree positive integers composed only of primes in $(a,b]$. 
We adopt the convention that $1\in \PP(a,b)$ for any $a,b$.

\begin{lem}\label{suma}
\begin{enumerate}
\item[(a)] For $t\ge w\ge 2$ and $k\ge 0$ we have
\[
\ssum{a\in \PP(w,t) \\ \omega(a)=k} \frac{1}{a} \le 
\frac{(\log_2 t-\log_2 w+O(1))^k}{k!}.
\]
\item[(b)] For $t\ge w\ge 2$ and $k\ge 1$ we have
\[
\ssum{a\in \PP(w,t) \\ \omega(a)=k} \frac{\log a}{a} \ll
\(1 + \log (t/w) \) 
\frac{(\log_2 t-\log_2 w+O(1))^{k-1}}{(k-1)!}.
\]
\item[(c)] For $2\le w\le s\le t$, we have
\[
\sum_{a\in \PP(w,t)} \frac{L(a)}{a} \ll \pfrac{\log t}{\log s}^2
\sum_{a\in \PP(w,s)} \frac{L(a)}{a}.
\]
\end{enumerate}
\end{lem}

\begin{proof}
Item (a) is immediate from
\[
\ssum{a\in \PP(w,t) \\ \omega(a)=k} \frac{1}{a} \le 
\frac{1}{k!} \Bigg(  \sum_{w<p\le t} \frac{1}{p} \Bigg)^k
\]
and Mertens' estimate.  For item (b), we have
\[
\ssum{a\in \PP(w,t) \\ \omega(a)=k} \frac{\log a}{a} =
\ssum{a\in \PP(w,t) \\ \omega(a)=k} \frac{1}{a} \sum_{p|a} \log p \le \sum_{w<p\le t} \frac{\log p}{p} \ssum{a\in \PP(w,t) \\ \omega(a)=k-1} \frac{1}{a}.
\]
The desired inequality follows from part (a) and Mertens' estimates.  For part (c), we factor each $a\in \PP(w,t)$
uniquely as $a=a_1a_2$ with $a_1\in \PP(w,s)$ and $a_2\in\PP(s,t)$.
Then, using Lemma \ref{Lbounds} (ii) we deduce that
\dalign{
\sum_{a\in \PP(w,t)} \frac{L(a)}{a} &\le \sum_{a_1\in \PP(w,s)} \frac{L(a_1)}{a_1} \sum_{a_2\in \PP(s,t)} \frac{\tau(a_2)}{a_2} \\
&= \prod_{s<p\le t} \(1+\frac{2}{p} \) \sum_{a_1\in \PP(w,s)} \frac{L(a_1)}{a_1}.
}
The desired inequality follows from Mertens' estimates.
\end{proof}

The following is a standard sieve bound, see e.g. \cite{HR}.
\begin{lem}\label{sieve}
(a)
 Uniformly for $x \ge 2z\ge 4$, we have
\[
|\{ x/2 <n\le x: P^-(n)>z \}| \gg \frac{x}{\log z}.
\]
Uniformly for $x\ge z\ge 2$ we have
\[
|\{ n\le x: P^-(n)>z \}| \ll \frac{x}{\log z}.
\]
\end{lem}

Finally, we quote standard bounds on the Poisson distribution,
see e.g. the results in Section 4 of \cite{Norton}.

\begin{lem}\label{Norton}
Uniformly for $h\le m\le x$, we have
\[
\sum_{h\le k\le m} \frac{x^k}{k!} \order 
\min\( \sqrt{x},\frac{x}{x-m}, m-h+1\) \frac{x^m}{m!}.
\]
\end{lem}

%
%
%
\section{Local-to-global estimates}

Following a kind of local-to-global principle
first utilized in \cite{F}, we bound $H(x,y,2y;\cR_w)$ in terms of the function $L(a)$.  This justifies the heuristic
presented in Section \ref{sec:heuristic}.

\begin{lem}\label{HL}
If $w\le y^{1/15}$ and  $y\le \sqrt{x}$, then
\[
H(x,y,2y;\cR_w)-H(x/2,y,2y;\cR_w) \gg \frac{x}{\log^2 y}\sum_{a\in \PP(w,y)} \frac{L(a)}{a}.
\]
If  $w \le y\le \sqrt{x}$ and $w \le y^{1/10}$, then 
\[
H(x,y,2y;\cR_w) \ll \frac{x}{\log^2 y} \sum_{a\in \PP(w,y)}  
\frac{L(a)}{a}.
\]
\end{lem}

\begin{proof}
We begin with the lower bound.
We may assume without loss of generality that $y\ge y_0$,
where $y_0$ is a sufficiently large constant, because in the
case $y<y_0$, for any prime $p\in (y,2y]$ (such $p$ exists by Bertrand's Postulate) and we see that
\[
H(x,y,2y;\cR_w)-H(x/2,y,2y;\cR_w) \gg x/p \gg_{y_0} x.
\]

Consider integers $n=ap_1 p_2b \in (x/2,x]$ with $P^-(a)>w$,
 $p_1$ and $p_2$ prime, satisfying the inequalities
\[ a\le y^{1/5} < p_1 < p_2 \le \frac14 y^{4/5} < P^-(b), \]
and with $\log(y/p_1p_2)\in \LL(a)$.  The last condition implies
that $\tau(ap_1p_2,y,2y)\ge 1$,
and we also have that $P^-(n)>w$.   Since
$y^{4/5} \le y/a < p_1p_2 \le 2y$, we have
$x/ap_1p_2 \ge x/(2y^{6/5}) \ge \frac12 y^{4/5}$.  Thus, by Lemma \ref{sieve},
for each triple $(a,p_1,p_2)$, the number of possible $b$ is
$\gg \frac{x}{ap_1 p_2\log y}$.  Now $\LL(a)$
is the disjoint union of intervals of length $\ge \log 2$
contained in $[-\log 2,\log a]$.  For each
such interval $[u,v)$, Mertens' estimate implies that
\[
 \ssum{u\le \log(y/p_1p_2)< v \\ y^{1/5}<p_1<p_2<\frac14 y^{4/5} } \frac{1}{p_1p_2}
 \ge \sum_{8y^{1/5} <p_1 < y^{2/5}} \frac{1}{p_1} \sum_{ye^{-v}/p_1 < p_2 \le ye^{-u}/p_1} \frac{1}{p_2} 
 \gg \frac{v-u}{\log y}.
\]
Here we made use of the estimate $v\le \log a \le \frac15 \log y$
which implies that $ye^{-v}/p_1 \ge y^{2/5}>p_1$.
Thus, with $a$ fixed, the sum of $\frac{1}{p_1p_2}$ is $\gg \frac{L(a)}{\log y}$ and we obtain
$$
H(x,y,2y;\cR_w)-H(x/2,y,2y;\cR_w) \gg \frac{x}{\log^2 y} \ssum{a\le y^{1/5} \\ P^-(a)>w} \frac{L(a)}{a}.
$$
We to replace the sum over $a$ with an unbounded set which is muliplicatively more convenient, starting with
\[
\ssum{a\le y^{1/5} \\ P^-(a)>w}  \frac{L(a)}{a} \ge \ssum{a\le y^{1/5} \\ a\in \PP(w,y^{1/15})}
 \frac{L(a)}{a} \ge \sum_{a\in \PP(w,y^{1/15})}\frac{L(a)}{a} \(1 - \frac{\log a}{\log(y^{1/5})}\).
\]
Break this into two sums, the first being what we want and the second involving
\[
\sum_{a\in \PP(w,y^{1/15})} \frac{L(a)\log a}{a} = 
\sum_{a\in \PP(w,y^{1/15})} \frac{L(a)}{a} \sum_{p|a} \log p = \sum_{w<p\le y^{1/15}} \frac{\log p}{p} 
 \ssum{b\in \PP(w,y^{1/15})\\ p\nmid b} \frac{L(pb)}{b}.
\]
Using the trivial relation $L(pb)\le 2L(b)$
which comes from Lemma \ref{Lbounds} (ii), and Mertens' estimate, we have
\[
 \ssum{a\le y^{1/5} \\ P^-(a)>w} \frac{L(a)}{a} \ge 
 \sum_{a\in \PP(w,y^{1/15})} \frac{L(a)}{a} \( 1 - 
 \frac{2\log(y^{1/15})+O(1)}{\log(y^{1/5})} \) \gg  
 \sum_{a\in \PP(w,y^{1/15})} \frac{L(a)}{a}.
\]
An application of Lemma \ref{suma} (c)
 concludes the proof of the lower bound.
 
For the upper bound, we first relate $H(x,y,2y;\cR_w)$ to $H^*(x,y,2y;\cR_w)$, the number of
\emph{squarefree} integers $n\le x$ with $P^-(n)>w$ and  $\tau(n,y,z)\ge 1$.  Write
$n=n'n''$, where $n'$ is squarefree, $n''$ is squarefull and $(n',n'')=1$.
The number of $n\le x$ with $n''>\log^{10} y$ is
$$
\le x \sum_{n''>\log^{10} y} \frac{1}{n''} \ll \frac{x}{\log^5 y}.
$$
If $n'' \le \log^{10} y$, then for some $f|n''$, $n'$ has a divisor in
$(y/f,2y/f]$, hence
\be\label{HH*}
H(x,y,2y;\cR_w)\le \ssum{n''\le \log^{10} y \\ P^-(n)>w} \;\;\sum_{f|n''} H^*\( \tfrac{x}{n''},
\tfrac{y}{f}, \tfrac{2y}{f};\cR_w \) + O\pfrac{x}{\log^5 y}.
\ee
Let $w_0$ be a sufficiently large absolute constant.  It suffices 
to prove the upper bound for $w\ge w_0$, for the 
case $w<w_0$ follows from the case $w=w_0$.
In the sum, \[
y/f \le y\le (x/n'')^{1/2} \log^{5} y \le (x/n'')^{5/9}
\]
for large enough $w_0$.
We will show that for $w_0\le y_1\le x_1^{5/9}$,
\be\label{e2}
H^*(x_1,y_1,2y_1;\cR_w)\ll x_1 \max_{t\ge y_1^{3/4}} \frac{1}{\log^2 t}
\sum_{a\in \PP(w,t)} \frac{L(a)}{a}.
\ee

It follows from \eqref{e2} and \eqref{HH*} that
\dalign{
H(x,y,2y;\cR_w) &\ll  \ssum{n''\le \log^{10} y \\ P^-(n)>w} \frac{x}{n''} \sum_{f|n''} 
\max_{t\ge (y/f)^{3/4}} \frac{1}{\log^2 t}
\sum_{a\in \PP(w,t)} \frac{L(a)}{a} \\
&\ll x \max_{t\ge y^{2/3}} \frac{1}{\log^2 t} \sum_{a\in \PP(w,t)} \frac{L(a)}{a} 
 \ssum{n''\le \log^{10} y \\ P^-(n)>w} \frac{\tau(n'')}{n''}. 
}
The lemma follows by noting that
the inner sum over squarefull $n''$ is $O(1)$, using the relative
estimate in Lemma \ref{suma} (c) with $s=y^{2/3}$,
and finally noting that $\PP(w,y^{2/3}) \subseteq \PP(w,y)$.

It remains to prove \eqref{e2}.  The right side is $\gg x_1/\log^2 y_1$ since $L(1)=\log 2$,
and hence it suffices to count those $n\in (x_1/\log^2 y_1,x_1]$.  We'll count
separately those $n\in (x_1/2^{r+1},x_1/2^r]$ for some integer $r$, $0\le r\le 5\log_2 y_1$.
Let $\AA$ be the set of squarefree integers $n\in (x_1/2^{r+1},x_1/2^r]$ 
with a divisor in $(y_1,2y_1]$.
Put $z_1=2y_1$, $y_2=\frac{x_1}{2^{r+2}y_1}$, $z_2=\frac{x_1}{2^r y_1}$.
If $n\in \AA$, then $n=m_1m_2$ with $y_i < m_i \le z_i$ ($i=1,2$).
For some $j\in \{1,2\}$ we have $p=P^+(m_j) < P^+(m_{3-j})$;
in particular, $p$ is not the largest prime factor of $n$.  Fixing $j$, we may
write $n=abp$, where $P^+(a) < p < P^-(b)$ and $b>p$.  
Since $\tau(ap,y_j,z_j)\ge 1$, we have $y_j/a \le p\le z_j$.
By Lemma \ref{sieve} and the fact that $b>p$,
given $a$ and $p$, the number of choices for $b$ is
\[
\ll \frac{x_1}{2^r ap\log p} \le \frac{x_1}{2^r ap\log \max\(P^+(a),y_j/a\)},
\]
Now $a$ has a divisor in $(y_j/p,z_j/p]$,
and thus $\log(y_j/p)\in \LL(a)$ or $\log(2y_j/p)\in \LL(a)$.  Since
$\LL(a)$ is the disjoint union of intervals of length $\ge \log 2$
with total measure $L(a)$, by repeated use of Mertens' estimate we obtain
\[
\sum_{\substack{\log (cy_j/p)\in \LL(a) \\ p\ge P^+(a)}} \frac{1}{p} \ll
\frac{L(a)}{\log \max\(P^+(a),y_j/a\)} \qquad (c=1,2).
\]
Since $y_j \ge y_1^{4/9}/2^{r+2} \ge y_1^{3/4}$, we have that
\[
H^*(x,y,2y;\cR_w) \ll \sum_{0\le r\le 5\log_2 y_1} \frac{x_1}{2^r} \sum_{t\in \{4y_1,4y_2\}}
\sum_{a\in \PP(w,t)}  \frac{L(a)}{a \log^2\(P^+(a)+t/(4a)\)}. 
\]
We have $4y_j \ge y_1^{4/5}/2^{r} \ge y_1^{3/4}$ for any $j$ and any $r$.
Also, by \cite[Lemma 2.2]{DK},
\[
\sum_{a\in \PP(w,t)} \frac{L(a)}{a \log^2\(t/(4a)+P^+(a)\)} \ll \frac{1}{\log^2 t}
\sum_{a\in \PP(w,t)} \frac{L(a)}{a}.
\]
Summing over $r$, we deduce \eqref{e2}.
 \end{proof}

%
%
%
%

\section{Proof of theorem \ref{mainthm}: lower bounds}\label{sec:lower}

We first deal with simple cases. 
Let $w_0$ be a sufficiently large constant and $\eps>0$
a sufficiently small constant.
Firstly, if $y\le w_0$, then
Bertrand's postulate implies that there is a prime $p\in(y,2y]$ and therefore
\[
H(x,y,2y;\cR_w)-H_z(x/2,y,2y;\cR_w) \ge \# \{ x/2 <n\le x: p|n \} \gg x.
\]
Also, if $w\le w_0 < y$ and $w\le y/8$, then
\[
H(x,y,2y;\cR_w)-H(x/2,y,2y;\cR_w) \ge H(x,y,2y;\cR_{w_0})-H(x/2,y,2y;\cR_{w_0})
\]
and the desired bound follows from the case $w=w_0$.
Thirdly, when $y > w_0$ and $y^{\eps} < w\le y/8$,
we consider two caess: (a) $y\le \sqrt{x/8}$ and (b) $\sqrt{x/8} < y\le \sqrt{x}$.  In case (a), consider $n=pm$ where $y<p\le 2y < P^-(m)$.  Since $x/p \ge 4y$ for all such $p$, 
Lemma \ref{sieve} implies
\dalign{
H(x,y,2y;\cR_w)-H(x/2,y,2y;\cR_w)  &\ge  \sum_{y<p\le 2y} \# \{ x/2p <n\le x/p: P^-(n)> 2y\} \\
&\gg \sum_{y<p\le 2y} \frac{x}{p\log y} \gg \frac{x}{\log^2 w}.
}
In the case  (b) $\sqrt{x/8} < y\le \sqrt{x}$, consider $n=pm$
where $y<p\le 2y$ and $P^-(m)>y/8$.  Such $n$ have at most three
prime factors larger than $y$, hence at most three representations in this form.  Since $x/p \ge 2y/8$, Lemma \ref{sieve} similarly implies that
\dalign{
H(x,y,2y;\cR_w)-H(x/2,y,2y;\cR_w)  &\ge  \frac13 \sum_{y<p\le 2y} \# \{ x/2p <n\le x/p: P^-(n)> y/8\} \\
&\gg \sum_{y<p\le 2y} \frac{x}{p\log y} \gg \frac{x}{\log^2 w}.
}

From now  on, we assume 
\be\label{yw-lrw}
w_0 < w\le y^{\eps}.
\ee

We begin with the
local-to-global estimate for $H(x,y,2y;\cR_w)$  given in 
Lemma \ref{HL}, and relate $L(a)$ to counts of \emph{pairs}
of divisors which are close together.
Evidently,
\be\label{LW}
L(a) \ge (\log 2) \# \{ d|a : \tau(a,d,2d)=0)\}
\ge (\log 2) (\tau(a) - W^*(a)),
\ee
where
\[
W^*(a) = \# \{ d|a,d'|a :  d < d' \le 2d \}.
\]

We will apply \eqref{LW} with integers whose prime
factors are localized.  As in \cite{F2},
partition the primes into  sets $D_1, D_2, \ldots$, where
each $D_j$ consists of the primes in an interval
$(\lam_{j-1},\lam_j]$, with $\lam_j \approx \lam_{j-1}^2$.  More precisely, 
let $\lam_0=1.9$ and define inductively $\lam_j$ for $j\ge 1$ as the
largest prime so that
\be\label{Dj} 
\sum_{\lam_{j-1} < p \le \lam_j} \frac{1}{p} \le \log 2.
\ee
For example, $\lam_1=2$ and $\lam_2=7$. 
By Mertens' bounds, we have
$$
\log_2 \lam_{j} -\log_2 \lam_{j-1} = \log 2 + O(1/\log \lam_{j-1}), 
$$
and it follows that for some absolute constant $K$,
\be\label{lam}
2^{j-K} \le \log \lam_j \le  2^{j+K} \qquad (j\ge 0).
\ee
For a vector $\bb=(b_1,\ldots,b_J)$ of non-negative integers,
let $\AA(\bb)$ be the set of square-free integers $a$ composed of exactly 
$b_j$ prime factors from $D_j$ for each $j$.

%
%

\begin{lem}\label{sumW}
Assume $\bb=(b_1,\ldots,b_{J_2})$, with $b_j=0$ for
$j<J_1$.  Then
$$
\sum_{a\in \AA(\bb)} \frac{W^*(a)}{a} \ll \frac{(\log 4)^{b_{J_1}+\cdots+b_{J_2}}}{b_{J_1}! 
  \cdots b_{J_2}!}  \sum_{j=J_1}^{J_2} 2^{-j+b_{J_1} + \cdots + b_j}.
$$
\end{lem}

\begin{proof}
Identical to the proof of Lemma 2.3 in \cite{F2}, except that
we remove the terms corresponding to $d=d'$.
\end{proof}

We will only consider those intervals $D_j \subseteq 
(w,y]$, that is, only $J_1 \le j\le J_2$, where
\[
J_1 :=  \min\{ j : \lam_{j-1}>w\}, \qquad
J_2 := \max\{ j : \lam_j \le y \}.
\]
By \eqref{lam}, we have
\be\label{J1J2-approx}
\bigg| J_1 - \frac{\log_2 w}{\log 2} \bigg| \le K+2,
\qquad \bigg| J_2 - \frac{\log_2 y}{\log 2} \bigg| \le K+1.
\ee
With $w_0$ sufficiently large, put
\be\label{Meps}
M = \frac{\log_2 w_0}{200}, \qquad \eps=2^{-200M-2K-4}.
\ee
In the sequel, all statements involving $M$ implicitly assume that
 $M$ be sufficiently large.
By \eqref{yw-lrw}, \eqref{J1J2-approx} and \eqref{Meps}, we have
\be\label{J1J2-lwr}
J_1 \ge 100M, \qquad J_2-J_1 \ge 100M.
\ee

Let $\BB_k$ be the set of vectors $(b_{J_1},\ldots,b_{J_2})$ which satisfy
\begin{enumerate}
\item[(a)] $b_{J_1} + \cdots + b_{J_2} = k$;
\item[(b)] $\sum_{j=J_1}^{J_2} 2^{-j + b_{J_1}+\cdots+b_j} \le 2^{-M}$;
\item[(c)] $b_{J_1+i-1} \le M+i^2 \;\; (i\ge 1)$;
\item[(d)] $b_{J_2-i+1} \le M+i^2 \;\; (i\ge 1)$.
\end{enumerate}

From the definition of $J_2$, whenever $\bb\in \BB_k$ 
and $a\in \AA(\bb)$, we have $a\in \PP(w,y)$.
By Lemma \ref{sumW}, for any $k$ and any $\bb\in \BB_k$ we have
\be\label{Wa} 
\sum_{a\in \AA(\bb)} \frac{W^*(a)}{a} \le \frac{1}{10} \frac{(\log 4)^k}{b_{J_1}!\cdots b_{J_2}!}.
\ee

 By \eqref{lam},  the fact that $J_1$ is sufficiently large,
  and $b_j \le (j+1-J_1)^2+M$, for any $k$ and $\bb\in \BB_k$ we have by \eqref{Dj}
\be\label{sumtau}
\begin{split}
\sum_{a\in\AA(\bb)} \frac{\tau(a)}{a} &= 2^k \prod_{j=J_1}^{J_2} 
\frac{1}{b_j!} \biggl(
  \sum_{p_1\in D_j} \frac{1}{p_1} \sum_{\substack{p_2\in D_j \\ p_2 \ne p_1}}
  \frac{1}{p_2} \cdots \sum_{\substack{p_{b_j}\in D_j \\ p_{b_j} \not\in
  \{ p_1, \ldots, p_{b_j-1} \} }} \frac{1}{p_{b_j}} \biggr) \\
&\ge 2^k \prod_{j=J_1}^{J_2} \frac{1}{b_j!}
   \biggl( \log 2 - \frac{b_j}{\lam_{j-1}} \biggr)^{b_j} 
\ge \frac{(\log 4)^k}{2 b_{J_1}! \cdots b_{J_2}!}.
\end{split}
\ee

Combining Lemma \ref{HL}, \eqref{LW}, \eqref{Wa}, and \eqref{sumtau}, we arrive at
\be\label{low1}
H(x,y,2y;\cR_w)-H(x/2,y,2y;\cR_w) \gg \frac{x}{\log^2 y} \sum_{k_1\le k\le k_2} \sum_{\bb\in \BB_k} \frac{(\log 4)^k}{b_{J_1}!\cdots b_{J_2}!}
\ee
for any $k_1\le k_2$.
We bound the sum on $\bb$ using techniques from \cite{F}.

Following our heuristic, take 
\be\label{k2}
k_2 = \flr{\min\(\frac{\log_2 y}{\log 2}, 2(\log_2 y-\log_2 w)\) - 2M}.
\ee
By \eqref{yw-lrw}, $k_2 \ge 100M$ and by
\eqref{J1J2-approx},
\[
k_2 = \min(J_2,(\log 4)(J_2-J_1))-2M+\theta, \quad 
|\theta| \le (\log 4)(2K+3).
\]
We will choose $k_1$ to satisy 
\be\label{k1}
10M \le k_1 \le k_2.
\ee
Also define
\be\label{vs}
v=J_2-J_1+1, \qquad \quad s=J_1-2-M.
\ee
Setting $g_i=b_{J_1+i-1}$ for $i\ge 1$, we have
\[
\sum_{i=1}^v 2^{-i+g_1+\cdots+g_i} = 2^{J_1-1} f(\bb) \le 2^{s+1}.
\]
By (c) and (d) in the definition of $\BB_k$, $g_i \le M+i^2$ and $g_{v+1-i}\le M+i^2$ for every $i\ge 1$.
Applying the argument on the top of page 419 in \cite{F}, it follows 
that for $k_1\le k\le k_2$ we have
\be\label{Yk-low}
 \sum_{\bb\in \BB_k} \frac{(\log 4)^k}{b_{J_1}!\cdots b_{J_2}!}
 \gg (v\log 4)^k \text{Vol} (Y_k(s,v)),
\ee
where $Y_k(s,v)$ is the set of $\bx=(\xi_1,\dots,\xi_k)\in \RR^k$ satisfying
\begin{enumerate} 
\item $0 \le \xi_1 \le \cdots \le \xi_k < 1$;
\item For $1\le i\le \sqrt{k-M}${\rm ,} $\xi_{M+i^2} > i/v$ and
$\xi_{k+1-(M+i^2)} < 1-i/v$;
\item $\sum_{j=1}^k 2^{j-v\xi_j} \le 2^s$.
\end{enumerate}

We now invoke a result from \cite{F} concerning the volume of
$Y_k(s,v)$.

\begin{lem}[{\cite[Lemma 4.9]{F}}]\label{Yk}
Uniformly for $v\ge 1${\rm ,} $10M \le k\le 100(v-1)${\rm ,} $s\ge M/2+1$ and 
$0\le k-v \le s-M/3-1$. 
Then
$$
\Vol(Y_k(s,v)) \gg \frac{k-v+1}{(k+1)!}.
$$
\end{lem}

Combining  \eqref{Meps}, \eqref{J1J2-lwr}, \eqref{k2}, \eqref{k1} and \eqref{vs}, we find
\dalign{
v &= J_2-J_1+1 \ge 100M, \\
10M &\le k_1 \le k_2 \le (\log 4)(J_2-J_1) = (\log 4)(v-1), \\
s &\ge \log_2 w - M \ge M/2+1, \\
k_2-v-s &\le (J_2-2M) - (J_2 - 1 - M) = 1-M \le -M/3-1.
}
Thus, the hypotheses of Lemma \ref{Yk} are satisfied
for all $k\in [k_1,k_2]$.
Therefore, gathering \eqref{low1}, \eqref{Yk-low} and invoking
Lemma \ref{Yk}, we conclude that
\be\label{H-lower-1}
H(x,y,2y;\cR_w)-H(x/2,y,2y;\cR_w)
\gg  \frac{x}{\log^2 y} 
\sum_{k_1 \le k\le k_2} \frac{(v\log 4)^k}{k!}\pfrac{k-v+1}{k+1}.
\ee

Consider three cases: 
I. $\delta \ge 1 - \frac{1}{\log 4}$, II. $\frac{1}{5} \le \delta < 1 - \frac{1}{\log 4}$, III. $0<\delta < \frac{1}{5}$.

\noindent
\textbf{Case I.} We have 
$\log_2 w \ge (1-\frac{1}{\log 4})\log_2 y$ and thus, by \eqref{k2} and \eqref{J1J2-approx},
\[
k_2 = \flr{2(\log_2 y-\log_2 w) - 2M} \ge v\log 4 - 2M-2K-4
\ge 1.1v.
\]
Now set $k_1 = 0.9 k_2$, so that \eqref{k1} is satisfied.
Then
\[
\frac{k-v+1}{k+1} \order 1 \qquad (k_1\le k\le k_2).
\]
Applying Lemma \ref{Norton} to the sum in
\eqref{H-lower-1}, we obtain
\[
H(x,y,2y;\cR_w)-H(x/2,y,2y;\cR_w) \gg\frac{x}{\log^2 y} 
e^{v\log 4} \gg \frac{x}{\log^2 w},
\]
as required in this case.

\medskip

\noindent
\textbf{Case II.}
By \eqref{k2}, followed by \eqref{J1J2-approx}, we have
\[
k_2 = \flr{ \frac{\log_2 y}{\log 2}- 2M}
\ge \frac{5}{4}(v-2K-4)-2M \ge 1.2v
\]
and take
\[
k_1 = \frac{9}{10} k_2.
\]
Thus,
\[
\frac{k-v+1}{k+1} \order 1 \order \delta \qquad
(k_1 \le k\le k_2).
\]
We apply Lemma \ref{Norton} with $h=k_1,m=k_2,x=v\log 4$,  and use
\dalign{
\min\(x^{1/2}, \frac{x}{x-m}, m-h+1\)&\gg \min\( (\log_2 y)^{1/2}, \frac{v\log 4}{v\log 4 - k_2} \) \\
&\gg \min\( (\log_2 y)^{1/2},  \frac{1}{1-\frac{1}{\log 4}-\delta}+O(1/\log_2 y) \)\\ &\gg \delta (\log_2 y)^{1/2} B(w,y).
}
Recalling the definition of $\cE$, we have by Stirling's formula,
\[
\frac{(v\log 4)^{k_2}}{k_2!} \gg \frac{(e(1-\delta))^{k_2}}{\sqrt{\log_2 y}} = \frac{(\log y)^{2-\cE + \frac{\log(1-\delta)}{\log 2}}}{\sqrt{\log_2 y}}.
\]
Invoking Lemma \ref{Norton} we see that the sum in \eqref{H-lower-1} is
\[
\gg \delta B(w,y) (\log y)^{2-\cE + \frac{\log(1-\delta)}{\log 2}},
\]
and
this gives the required lower bound in Theorem \ref{mainthm}.

\medskip

\noindent
\textbf{Case III.}  When $\delta < \frac{1}{10}$, we also have
\[
k_2 = \flr{ \frac{\log_2 y}{\log 2}- 2M} = J_2-2M+O(K),
\]
but in this case we take
\[
k_1 = k_2,
\]
as we are in the range where the summation in \eqref{H-lower-1}
is dominated by the final summand regardless of $k_1$.
Here
\[
k_2-v+1 \order J_1,\qquad \frac{k_2-v+1}{k_2+1} \order \frac{\log_2 w}{\log_2 y} = \delta.
\]
Applying Lemma \ref{Norton} to the sum in \eqref{H-lower-1},
we obtain
\[
H(x,y,2y;\cR_w)-H(x/2,y,2y;\cR_w)
\gg  \frac{\delta x}{\log^2 y} \, \frac{(v\log 4)^{k_2}}{k_2!}.
\]
Applying Stirling's formula as in Case II and observing that
$B(w,y)=1$ in this case, we conclude the desired upper bound.

This completes the proof of the lower bound in Theorem \ref{mainthm}.

%
%
%
%
\section{Proof of Theorem \ref{mainthm}: upper bounds}\label{sec:upper}
%
%

In this section, we prove the upper bound in Theorem \ref{mainthm}.
We begin with simple cases.  If $w_0$ is fixed and $w\le w_0$, then
$H(x,y,2y;\cR_w) \le H(x,y,2y)$ and the required bound follows from \eqref{Hxy2y}.  Next, if $\log_2 w \ge (1-1/\log 4)\log_2 y$, then by Lemma \ref{sieve},
\[
H(x,y,2y;\cR_w) \le \ssum{y<d\le 2y \\ P^-(d)>w}
| \{m\le x/d:  P^-(m)>w\}|  \ll
\ssum{y<d\le 2y \\ P^-(d)>w} \frac{x}{d\log w} \ll \frac{x}{\log^2 w},
\]
as required.

From now on, we assume that
\be\label{wy}
\log w \le  (\log y)^{1-1/\log 4},
\ee
that is, $\delta \le 1-\frac{1}{\log 4}$.
We apply Lemma \ref{HL} and use
upper bounds for $L(a)$ from Lemma \ref{Lbounds}.  As in \cite{F}, the sums
involving $L(a)$ are bounded in terms of multivariate integrals,
which were estimated accurately in \cite{F,F2}.

\subsection{Case I. $\frac{1}{10} \le \delta \le 1-\frac{1}{\log 4}$.}

This case is very easy, as we expect no clustering of divisors.  Let
\be\label{k0}
k_0 = \flr{\frac{\log_2 y}{\log 2}}.
\ee
Beginning with Lemma \ref{HL}, we apply Lemma \ref{Lbounds} (i) to
bound $L(a)$ and then apply  
Lemma \ref{suma} parts (a) and (b).  We have
\dalign{
H(x,y,2y;\cR_w) &\ll \frac{x}{\log^2 y} \Bigg[
\sum_{k\le k_0} 2^k \ssum{a\in \PP(w,y) \\ \omega(a)=k} \frac{1}{a}
+ \sum_{k> k_0} \ssum{a\in \PP(w,y) \\ \omega(a)=k} \frac{\log a}{a}
\Bigg] \\
&\ll   \frac{x}{\log^2 y} \Bigg[ \sum_{k\le k_0} \frac{(2\log_2 y-2\log_2 w)^k}{k!} + (\log y) \sum_{k\ge k_0} \frac{(\log_2 y-\log_2 w)^k}{k!} \Bigg].
}
Since $k_0 \ge 1.4 (\log_2 y-\log_2 w)$, the second sum 
on the right side is dominated by the single term $k=k_0$
and thus by Stirling's formula we get that
\[
\sum_{k\ge k_0} \frac{(\log_2 y-\log_2 w)^k}{k!} \ll 
\frac{(\log_2 y-\log_2 w)^{k_0}}{k_0!} \ll
\frac{((e\log 2)(1-\delta))^{k_0}}{(\log_2 y)^{1/2}} \ll
\frac{(\log y)^{1-\cE + \frac{\log(1-\delta)}{\log 2}}}{(\log_2 y)^{1/2}}.
\]
We have $k_0 \le 2\log_2 y-2\log_2 w$ in the first sum,
for which we invoke Lemma \ref{Norton} and obtain,
with $\alpha = \log_2 y-\log_2 w$ the bound
\dalign{
\sum_{k\le k_0} \frac{(2\log_2 y-2\log_2 w)^k}{k!} &\ll
\frac{(2\alpha)^{k_0}}{k_0!} 
\min\( \alpha^{1/2}, \frac{\alpha}{\alpha-k_0} \) \\
&\ll (2e (\log 2)(1-\delta))^{k_0} \min\(1, (\log_2 y)^{-1/2}
((1-\delta)\log 4-1)^{-1}\) \\
&\ll (\log y)^{-\cE + \frac{\log(1-\delta)}{\log 2}} 
B(w,y),
}
as required for Theorem \ref{mainthm}.

\medskip

\subsection{Case II. $\delta \le \frac{1}{10}$.}
This case is more delicate, because we expect that typically
there will be clustering of
the divisors of $a$, and we must bound the probability of 
non-clustering.
 
We cut up the sum in Lemma \ref{HL} according to $\omega(a)$.  Let
$$
T_k = \sum_{\substack{a\in \PP(w,y) \\ \omega(a)=k}} \frac{L(a)}{a}.
$$

%
%

We bound $T_k$ in terms of a mutivariate integral, in a manner similar to 
that in \cite{F2}.

\begin{lem}\label{Tnint}
Suppose $w$ is large, \eqref{wy} holds,
let 
\[
v=\flr{\frac{\log_2 y-\log_2 w}{\log 2}}, \quad
u=\flr{\frac{\log_2 w}{\log 2}}
\]
and assume that $1\le k\le 10v$.  Then
\[
T_k \ll (2\log_2 y-2\log_2 w)^k U_k(v,u), 
\]
where
\[
U_k(v,u)= \idotsint\limits_{0\le \xi_1 \le \cdots\le \xi_k\le 1} 
\min_{0\le j\le k} 2^{-j} (2^{v\xi_1+u}+\cdots+2^{v\xi_j+u}+1)\, d\bx.
\]
\end{lem}

\begin{proof}
The proof is the same as the proof of Lemma 3.5 in \cite{F2}, except that we make use of the
fact that $P^-(a)>w$.  Recall the definition of 
the sets $D_j$ from Section \ref{sec:lower}.  By \eqref{lam}, any prime divisor of $a$ lies in $D_j$ 
with $ u -K-2 \le j\le v+u+K+3$. 
Following the proof of \cite[Lemma 3.5]{F2},
in particular using Lemma \ref{Lbounds} (iii), we have 
\dalign{
T_k \ll \frac{(2\log 2)^k}{k!} \int_{[u-K-2,v+u+K+4]^k} F(\mathbf{t})\,  d\mathbf{t},
}
where, letting $s_1\le s_2 \le \cdots \le s_k$  be the increasing rearrangement of
$t_1,\ldots,t_k$,
\[
F(\mathbf{t}) = \min_{0\le j\le k} 2^{-j} (2^{s_1}+\cdots+2^{s_j}+1).
\]
Observe that $F(\mathbf{t})$ is symmetric in $t_1,\ldots,t_k$.  Making the
change of variables 
\[
t_i=u-K-2+(v+2K+6)\xi_i \qquad (1\le i\le k)
\]
we see that 
$0\le \xi_i\le 1$ for each $i$.  Utilizing the summetry of $F(\mathbf{t})$, we find that
the multiple integral on the right side equals
\[
 (v+2K+6)^k k! \idotsint\limits_{0\le \xi_1 \le \cdots \le \xi_k \le 1} \min_{0\le j\le k} 
2^{-j} \(2^{(v+2K+6)\xi_i+u}+ \cdots + 2^{(v+2K+6)\xi_g+u} + 1\) d\boldsymbol{\xi}.
\]
We conclude that
\[
T_k(y) \ll ((2\log 2)(v+2K+6))^k U_k(v,u).
\]
Lastly, $(v+2K+6)^k \ll v^k$ since $k\le 10v$, and the
lemma follows.
\end{proof}

To bound $U_k(u,v)$ we invoke the following estimate from 
\cite{F,F2}.

\begin{lem}[{\cite[Lemma 13.2]{F}},{\cite[Lemma 4.4]{F2}}]\label{UUlem}
Define
\[
\TT(k,v,\g) = \{ \bx\in \RR^k : 0 \le \xi_1 \le \cdots \le \xi_k \le 1, 
  2^{v\xi_1} + \cdots + 2^{v\xi_j} \ge 2^{j-\g}\; (1\le j\le k) \}.
\]
Suppose $k,v,\gamma\in \ZZ$  with $1\le k\le 10v$ and $\g\ge 0$.  
Set $b=k-v$.  Then
$$
\Vol(\TT(k,v,\g)) \ll \frac{Y}{2^{2^{b-\g}} (k+1)!}, \qquad
Y = \begin{cases} b & \text{ if } b\ge \g+5 \\ (\g+5-b)^2(\g+1) & 
\text{ if } b < \g+5  \end{cases}.
$$
\end{lem}

\begin{lem}\label{Unlem}
Suppose $k,u,v$ are integers satisfying $1\le k\le 10v$ and $u\ge 1$.  Then
$$
U_k(v,u) \ll \frac{u(1 + |k-v-u|^2)} {(k+1)! (2^{k-v-u}+1)}.
$$
\end{lem}

Notice that the bound in Lemma \ref{Unlem} undergoes a change of behavior
at $k = v+u$.  

\begin{proof}
Put $b=k-v$.
For integers $m\ge 0$,  consider $\bx$ satisfying
\[
2^{-m} \le\min_{0\le j\le k} 2^{-j}\( 2^{v\xi_1+u}+\cdots+2^{v\xi_j+u}+1\)  < 2^{1-m}.
\]
 For $1\le j\le k$ we have
\[
2^{-j} \( 2^{v\xi_1+u}+\cdots+2^{v\xi_j+u} \) \ge
\max(2^{-j}, 2^{-m-u}-2^{-j-u}) \ge 2^{-m-u-1},
\]
and thus $\bx \in \TT(k,v,m+u+1)$.  Invoking Lemma \ref{UUlem}, we find that
\begin{align*}
U_k(v,u)  &\le \sum_{m\ge 0} 2^{1-m} \Vol(\TT(k,v,m+u+1))
\ll \frac{1}{(k+1)!} \sum_{m\ge 0} \frac{2^{-m} Y_m}
  {2^{2^{b-m-u-1}}}, \\
Y_m &= \begin{cases} b & \text{ if } m+u \le b-6 \\ (m+u+6-b)^2(m+u+2) & 
\text{ if } m+u > b-6 \end{cases}.
\end{align*}
Dividing the sum according to the two cases yields
$$
\sum_{m\ge 0} \frac{2^{-m} Y_m}{2^{2^{b-u-m-1}}} \ll \!\!\!
\sum_{0\le m< b-u-5} \frac{b}{2^m 2^{2^{b-m-u-1}}} + \!\! \sum_{m\ge 
\max(0,b-u-5)}\!\! \frac{(m+u+6-b)^2(m+u+2)}{2^m}.
$$
The proof is completed by noting that
if $b\ge 6+u$, each sum on the right side is $\ll b 2^{u-b}$ and
if $b\le 5+u$, the first sum is empty and the second is 
$\ll (6+u-b)^2 \ll 1+(b-u)^2$.
\end{proof}

Finally, we complete the upper bound in Theorem \ref{mainthm}.
Let $v=\flr{\frac{\log_2 y-\log_2 w}{\log 2}}$,
$u=\flr{\frac{\log_2 w}{\log 2}}$ and define $k_0$ by \eqref{k0}. 
Note that $k_0=v+u+O(1)$. We now combine Lemmas \ref{Tnint}
and \ref{Unlem}.  Since $k_0 > 1.4 (\log_2 y-\log_2 w)$, we have
\dalign{
\sum_{k_0 \le k\le 10k_0} T_k &\ll  \sum_{k_0 \le k\le 10k_0} 
\frac{u(1+(k-k_0)^2)}{(k+1)! 2^{k-u-v}} (2\log_2 y-2\log_2 w)^k \\
&\ll u 2^{k_0} \sum_{\ell\ge 0} \frac{1+\ell^2}{(k_0+1+\ell)!}
(\log_2 y-\log_2 w)^{k_0+\ell} \\
&\ll (\log_2 w) \frac{(2\log_2 y-2\log_2 w)^{k_0}}{(k_0+1)!}.
}
Similarly, since $k_0 \le 0.9(2\log_2 y-2\log_2 w)$, we have
\dalign{
\sum_{0\le k < k_0} T_k &\ll 1+\sum_{1\le k < k_0}
\frac{u (k_0-k)^2 (2\log_2 y-2\log_2 w)^k}{(k+1)!} \\
&\ll 1+ u \sum_{\ell=1}^{k_0-1} \frac{u \ell^2 (2\log_2 y-2\log_2 w)^{k_0-\ell}}{(k_0+1-\ell)!} \\
&\ll  (\log_2 w) \frac{(2\log_2 y-2\log_2 w)^{k_0}}{(k_0+1)!}.
}
For the large values of $k$ we use the crude bound
$L(a) \ll \tau(a)$ from  Lemma \ref{Lbounds} (i), followed by an
application of Lemma \ref{suma} (a).  This gives
\begin{align*}
\sum_{k\ge 10k_0} T_k &\le \sum_{k\ge 10k_0} \sum_{\substack{a\in\PP(w,y)\\ \omega(a)=k}} \frac{2^k \log 2}{a}
\le \sum_{k\ge 10k_0} \frac{(2\log_2 y-2\log_2 w+O(1))^k}{k!} \\
&\ll \frac{(2\log_2 y-2\log_2 w+O(1))^{10 k_0}}{(10k_0)!} \\
&\ll \frac{(2\log_2 y-2\log_2 w)^{k_0}}{(k_0+1)!}.
\end{align*}
Combining these three bounds for sums of $T_k$ with Lemma \ref{HL}, Lemma \ref{Norton}, and Stirling's formula, we conclude that
\dalign{
H(x,y,2y;\cR_w) &\ll \frac{x}{\log^2 y}(\log_2 w)  \frac{(2\log_2 y-2\log_2 w)^{k_0}}{(k_0+1)!} \\
&\ll \frac{x\log_2 w}{(\log_2 y)^{3/2}} (\log y)^{-\cE + \frac{\log(1-\delta)}{\log 2}}.
}
The proof of the upper bound in Theorem \ref{mainthm} is complete.

%
%

\bibliographystyle{amsplain}
\bibliography{h-rough}

\end{document}